\documentclass[letterpaper]{amsart}

\usepackage{amssymb}
\usepackage{color}
\usepackage{graphicx}
\usepackage{ifthen}
\usepackage{mathrsfs} 
\usepackage{mathscinet} 
\usepackage{mathtools}
\usepackage{rotating}
\usepackage{stmaryrd}
\usepackage{linsys}
\usepackage{braket}
\usepackage{hyperref}
\usepackage{centernot}
\usepackage{tikz}

\usepackage[numbers,sort&compress]{natbib}

\newcommand*{\Z}{\mathbb{Z}}
\newcommand*{\Q}{\mathbb{Q}}
\newcommand*{\R}{\mathbb{R}}

\DeclarePairedDelimiter{\verts}{\lvert}{\rvert}

\newcommand*{\braces}[1]{\left\lbrace #1 \right\rbrace}
\newcommand*{\parens}[1]{\left\lparen #1 \right\rparen}

\newcommand*{\card}{\verts}
\newcommand*{\setof}{\braces}

\newcommand*{\deftobe}{\mathrel{\coloneqq}}

\newcommand*{\maps}{\colon}
\newcommand*{\st}{\,:\,}   % "such that"

\DeclareMathOperator{\conv}{Conv}

\renewcommand*{\epsilon}{\varepsilon}
\renewcommand*{\phi}{\varphi}
\renewcommand{\subset}{\subseteq}

\newtheorem{theorem}{Theorem}[section]
\newtheorem*{theorem*}{Theorem}

\newtheorem*{lemma*}{Lemma}

\newtheorem{proposition}[theorem]{Proposition}
\newtheorem*{proposition*}{Proposition}

\newtheorem*{prob*}{Problem}

\newtheorem*{claim*}{Claim}

\theoremstyle{definition}

\newtheorem{definition}[theorem]{Definition}

\theoremstyle{remark}

\renewcommand{\P}{\mathcal{P}}

\newcommand{\QQ}{\mathcal{Q}}
\newcommand{\RR}{\mathcal{R}}
\newcommand{\Qhat}{\widehat{\Q}}
\newcommand{\Qstar}{Q_{\ast}}

\DeclareMathOperator{\Vol}{Vol}
\newcommand*{\ehr}{\operatorname{ehr}}  % Ehrhart quasi-polynomial
\newcommand{\pt}[1]{\mathbf{#1}}

\DeclareMathOperator{\Pyr}{\Delta}

\newcommand{\divides}{\mathbin{\mid}}

\title{%
   Rational polytopes with Ehrhart coefficients of arbitrary 
   period
}%

\author[T.~B.~McAllister]
{%
   Tyrrell B. McAllister
}
\address
{%
   Department of Mathematics \\
   University of Wyoming \\
   Laramie, WY 82071 \\
   USA
}
\email
{%
   tmcallis@uwyo.edu
}

\subjclass[2010]{Primary 52B05; Secondary 05A15, 52B11, 52C07}
\keywords{Ehrhart quasi-polynomials; Period sequences; Integer
lattice points; Rational polytopes; Cyclic polytopes}

\begin{document}

\begin{abstract}
   A seminal result of E.~Ehrhart states that the number of
   integer lattice points in the dilation of a rational polytope
   by a positive integer $k$ is a quasi-polynomial function of~$k$
   --- that is, a ``polynomial'' in which the coefficients are
   themselves periodic functions of $k$.  Using a result of F.~Liu
   on the Ehrhart polynomials of cyclic polytopes, we construct
   not-necessarily-convex rational polytopes of arbitrary
   dimension in which the periods of the coefficient functions
   appearing in the Ehrhart quasi-polynomial take on arbitrary
   values.
\end{abstract}

\maketitle

\section{Introduction}

The Ehrhart function $\ehr_{\P}$ of an $n$-dimensional rational
polytope $\P \subset \R^{n}$ counts the number of integer lattice
points in the $k$\textsuperscript{th} dilate of $\P$.  That is,
$\ehr_{\P}(k) = \card{k\P \cap \Z^{n}}$ for integers $k \ge 1$.
It is well known that $\ehr_{\P}(k)$ is a degree-$n$
quasi-polynomial function of $k$, meaning that
\begin{equation}
   \label{eq:EhrhartQP}
   \ehr_{\P}(k) = \sum_{i=0}^{n} c_{i}(k) k^{i}, \qquad
   \text{for $k \in \Z_{\ge 1}$,}
\end{equation}
where the \emph{coefficient functions} $c_{i}\maps \Z \to \Q$ are
periodic functions with finite periods.  In other words, writing
$\Qhat$ for the ring of periodic functions $\Z \to \Q$, each
rational polytope $\P$ has an associated \emph{Ehrhart
quasi-polynomial} $\ehr_{\P}(t) \in \Qhat[t]$.  (We refer the
reader to~\cite{BecRob2007, Hib1992a, Sta1997} for introductions
to Ehrhart theory.)

The purpose of this paper is to study the possible periods of the
coefficient functions $c_{i}$ appearing in equation
\eqref{eq:EhrhartQP}.  Our main result (Theorem~\ref{thm:Main}
below) is that \emph{these periods may take on arbitrary values}.
Thus, we offer a contribution to the project of characterizing the
Ehrhart quasi-polynomials of all rational polytopes.  This latter
project has been the subject of a great deal of work for several
decades.  Many deep constraints on the coefficients $c_{i}$ have
been found.  See in particular~\cite{Sta1980, Sta1991, BetMcM1985,
Hib1994, BDLDPS05, Bra2008, Bre2012, HNP2009, HenTag2009, Pfe2010,
Sta2016, Sta2009, Her2010msthesis} and references therein.

It is a remarkable and humbling fact that many of these famous
results do not make full use of \emph{convexity} (cf.~\cite[Remark
1.12]{Sta2016}).  That is, the constraints discovered are
satisfied by types of rational polytopal balls that are more
general than just the convex polytopes, such as the star convex
polytopes.  Here, by a \emph{rational polytopal ball}, or a
\emph{not-necessarily-convex} rational polytope, we mean a
topological ball in $\R^{n}$ that is a union $\bigcup_{i \in I}
\P_{i}$ of a finite family $\setof{\P_i \st i \in I}$ of convex
rational polytopes, all with the same affine span, in which every
nonempty intersection $\P_{i} \cap \P_{j}$, $i \ne j$, is a common
facet of $\P_{i}$ and~$\P_{j}$.  Only in dimension $n=2$ do we
have a complete characterization of the Ehrhart polynomials of
precisely the \emph{convex} integral polygons \cite{Sco1976}.
Even here in dimension $2$, the case of \emph{nonintegral} convex
polygons remains open \cite{McAMor2017, Her2010msthesis}.

Nonetheless, even in the not-necessarily-convex case, the complete
characterization of Ehrhart quasi-polynomials of rational polytope
still seems quite far off.  Chastened by the difficulty of such a
complete characterization, we restrict our attention in this paper
and its predecessors (\cite{McAMor2017, McARoc2018}) to the
\emph{periods} of the coefficient~functions $c_{i}$.  To this end,
define the \emph{period sequence} of $\P$ to be the sequence
$(p_{0}, p_{1}, \dotsc, p_{n})$ in which $p_{i}$ is the (minimum)
period of $c_{i}$.  That is, $p_{i}$ is the minimum positive
integer such that $c_{i}(k) = c_{i}(k+p_{i})$ for all $k \in \Z$.
Our motivating question is thus: \emph{What are the possible
period sequences of rational polytopes?}

It is well known that, if $\P \subset \R^{n}$ is $n$-dimensional,
then the leading coefficient of $\ehr_{\P}(t)$ is the volume of
$\P$.  In particular, $c_{n}$ is a constant, so $p_{n} = 1$.
A~result of Beck, Sam, and Woods~\cite{BecSamWoo2008} provides a
polytope with period sequence $(p_{0}, \dotsc, p_{n-1}, 1)$,
provided that the desired periods $p_{i}$ satisfy the divisibility
relations $p_{n-1} \divides p_{n-2} \divides \dotsb \divides
p_{0}$.  In particular, the polytopes constructed in
\cite{BecSamWoo2008} all satisfy $p_{0} \ge p_{1} \ge \dotsb \ge
p_{n}$.  Polytopes can fail to satisfy these inequalities when
they exhibit the phenomenon of \emph{period collapse}
\cite{McAMor2017}.  Nonetheless, the construction
in~\cite{BecSamWoo2008} gives convex rational polytopes with
arbitrary period sequences of the form $(p, 1, \dotsc, 1)$.  (See
Theorem~\ref{thm:BecSamWoo2008} below.)

A polytope $\P \subset \R^{n}$ is \emph{integral} if all of its
vertices lie in the integer lattice $\Z^{n}$.  In this case,
$\ehr_{\P}(t)$ is simply a polynomial.  That is, the period
sequence of an integral polytope is $(1, \dotsc, 1)$.
In~\cite{McARoc2018}, we constructed convex rational polytopes
with arbitrary period sequences of the form $(1, p, 1,\dotsc, 1)$.
It is straightforward to glue the constructions
in~\cite{BecSamWoo2008} and~\cite{McARoc2018} along a common
integral facet to form a convex rational polytope establishing the
following.
\begin{theorem}
\label{thm:pp111}
   Let positive integers $p_{0}$ and $p_{1}$ be given.  Then there
   exists a convex $n$-dimensional polytope with period sequence
   $(p_{0}, p_{1}, 1, \dotsc, 1)$.
\end{theorem}

Controlling the periods of higher-degree coefficients proved to be
more difficult.  In~\cite{McARoc2018}, we were able to exploit
previously discovered solutions to the system of Diophantine
equations known as the \emph{ideal Prouhet--Tarry--Escott} (PTE)
problem~\cite{Bor2002} to find $n$-dimensional polytopal balls
with period sequences of the form $(1, \dotsc, 1, p, 1)$, provided
that the dimension $n$ satisfied either \mbox{$3 \le n \le 11$} or
$n = 13$.

The main result of the current paper supersedes the PTE-based
construction from~\cite{McARoc2018} by proving the existence of
not-necessarily-convex polytopes of arbitrary dimension $n$ with
arbitrary period sequences $(p_{0}, p_{1}, \dotsc, p_{n-1}, 1)$.
\begin{theorem}[\protect{Proved in Section~\ref{sec:pppp1}}]
   \label{thm:Main}
   Let positive integers $p_{0}, \dotsc, p_{n-1}$ be given.  Then
   there exists an $n$-dimensional polytopal ball $\Qstar$ such
   that the period of the coefficient of $t^{i}$ in
   $\ehr_{\Qstar}(t)$ is $p_{i}$ for $0 \le i \le n-1$.
\end{theorem}

The proof of Theorem~\ref{thm:Main} depends upon a remarkable
property of cyclic polytopes (Theorem \ref{thm:Liu} below).  We
recall these polytopes and their Ehrhart polynomials in
Section~\ref{sec:CyclicPolytopes}.  In
Section~\ref{sec:BuildingBlocks}, we introduce the notation and
basic building blocks that we will use in our constructions.  In
Section~\ref{sec:11p11}, we build a rational polytopal ball with a
period sequence of the form $(1, \dotsc, 1, p, 1, \dotsc, 1)$, in
which a coefficient function of arbitrary degree has arbitrary
period.  Finally, in Section~\ref{sec:pppp1}, we combine the
constructions from Section~\ref{sec:11p11} to build a polytopal
ball with an arbitrary period sequence of the form $(p_{0},
p_{1},\dotsc, p_{n-1}, 1)$.

\section{Cyclic polytopes}
\label{sec:CyclicPolytopes}

Cyclic polytopes are perhaps most famous for their appearance in
the Upper Bound Theorem (McMullen~\cite{McM1970}): A
$d$-dimensional cyclic polytope attains the maximum number of
faces of every dimension among all $d$-dimensional polytopes with
the same number of vertices.  However, it is the Ehrhart
polynomials of cyclic polytopes, rather than their face lattices,
that will be of particular interest to us.

We recall the definition of cyclic polytopes.  Fix a subset $T
\subset \Z$ of $n+1$ integers.  (The particular subset chosen will
not matter for our purposes.)
% (For concreteness, we may take $T
% \deftobe \setof{0, 1, \dotsc, n}$.)  
We define a sequence of
polytopes $C_{i} \subset \R^{i}$, with $0 \le i \le n$, as
follows.  Let $C_{0} \deftobe \setof{0} \subset \R^{0}$.  For $1
\le i \le n$, let $\chi_{i} \maps \R \to \R^{i}$ be the moment
curve $x \mapsto (x, x^{2}, \dotsc, x^{i})$.  Then the
\emph{cyclic polytope} $C_{i} \subset \R^{i}$ is the convex hull
of the image of $T$ under $\chi_{i}$.  That is, $C_{i} \deftobe
\conv(\chi_{i}(T))$.

The Ehrhart polynomials of cyclic polytopes are unusual in that
all of their coefficients have straightforward geometric
interpretations.  Such interpretations are always available for
the two leading coefficients, $c_{d}$ and $c_{d-1}$, of the
Ehrhart polynomial of an arbitrary $d$-dimensional integral
polytope $\P$.
% It is well known that, in general, if $\P$ is a
% $d$-dimensional integral polytope, then the two leading Ehrhart
% coefficients $c_{d}$ and $c_{d-1}$ of~$\ehr_{\P}(t)$ are easily
% expressed in terms of the volume of $\P$ and the volumes of its
% facets, respectively.  
However, in the general case, no such interpretations exist for
the lower-degree coefficients of $\ehr_{\P}(t)$.  The cyclic
polytopes are a striking exception.  In particular,
F.~Liu~\cite{Liu2005} proved that
% (regardless of the initial
% $(n+1)$-subset $T \subset \Z$ chosen) 
the Ehrhart polynomial of $C_{i}$ satisfies a beautiful recursive
expression first conjectured by Beck~et~al.\ in~\cite{BDLDPS05}:
\begin{theorem}[Liu~\cite{Liu2005}]
\label{thm:Liu}
   The Ehrhart polynomials of the cyclic polytopes are given by
   \begin{equation}\label{eq:CyclicRecurrence}
      \ehr_{C_{i}}(t) = \Vol(C_{i})t^{i} + \ehr_{C_{i-1}}(t),
      \qquad \text{for $1 \le i \le n$},
   \end{equation}
   or, equivalently,
   \begin{equation*}
      \ehr_{C_{i}}(t) = \sum_{j=0}^{i} \Vol(C_{j})t^{j},
   \end{equation*}
   where $\Vol(C_{j})$ denotes the volume of $C_{j}$ in $\R^{j}$.
   (By convention, $C_{0}$ has volume $1$.)
\end{theorem}
We remark that the known proofs of this elegant result are far
from trivial~\cite{Liu2005, Liu2009, Liu2011}.

\section{Notation and building blocks}
\label{sec:BuildingBlocks}

In this section, we briefly review notation and results developed
in~\cite[Sections 2 and 3]{McARoc2018}, to which we refer the
reader for additional discussion and examples.

Our goal is to build polytopes of arbitrary dimension with
arbitrary prescribed period sequences.  Since adding a polynomial
to a quasi-polynomial does not change the period sequence, we will
consider two quasi-polynomials to be equivalent if their
difference is a polynomial.  Recall that we write $\Qhat$ for the
ring of periodic functions~$\Z \to \Q$.

\begin{definition}
   Two quasi-polynomials $q(t), r(t) \in \Qhat[t]$ are
   \emph{equivalent} if $q(t) - r(t) \in \Q[t]$.  In this case, we
   write $q(t) \equiv r(t)$.
\end{definition}

The chief convenience of this notation is that, if $\QQ \cup \RR$
is a union of rational polytopes $\QQ$ and $\RR$ such that $\QQ
\cap \RR$ is integral, then $\ehr_{\QQ \cup \RR} (t) \equiv
\ehr_{\QQ}(t) + \ehr_{\RR}(t)$.  Since $\Q[t]$ is not an ideal in
the ring $\Qhat[t]$, care must taken when multiplying
quasi-polynomials.  Nonetheless, a limited kind of substitution
holds: if $f(t) \in \Q[t]$ and $q(t) \equiv r(t) \in \Qhat[t]$,
then $f(t) q(t) \equiv f(t)r(t)$.

Fix a positive integer $p$.  (Typically, $p$ will be the desired
period of a coefficient function in the Ehrhart quasi-polynomial
of a rational polytope.)  Our constructions begin with two
fundamental building blocks: the closed line segment $\ell
\deftobe [-\frac{1}{p}, 0] \subset \R$, and the convex
pentagon\footnote{When $p=1$, $P$ is a triangle.} $P$ in $\R^{2}$
with vertices $\pt{u}^{+}$, $\pt{u}^{-}$, $\pt{v}^{+}$,
$\pt{v}^{-}$, $\pt{w}$, where
\begin{align}
   \pt{u}^{\pm} & \deftobe \pm q \pt{e}_{1}, &
   \pt{v}^{\pm} & \deftobe \pm (q-1) \pt{e}_{1} + \pt{e}_{2} , &
   \pt{w} & \deftobe \frac{q}{p}\pt{e}_{2}, \label{eq:pentagon}
\end{align}
and $q := p^{2} - p + 1$.  (Here and below, we write $\pt{e}_{i}$
for the $i$th standard basis vector.) 

A key fact, proved in~\cite{McAMor2017}, is that the Ehrhart
quasi-polynomials of $P$ and $\ell$ are ``complements'' of each
other in the sense that the periodic parts of their coefficients
cancel when the quasi-polynomials are added together.  That is,
\begin{equation}
   \label{eq:PentagonEquivToSegment}%
   \ehr_{P}(t) \equiv -\ehr_{\ell}(t).
\end{equation}
Furthermore, this equivalence is respected by the operation of
taking \mbox{$i$-fold} pyramids over $P$ and $\ell$.  The
\emph{pyramid} $\Pyr(\QQ)$ over a polytope $\QQ \subset \R^{d}$ is
the convex hull of the embedded copy of $\QQ$ in $\R^{d+1}$ at
height $0$ together with the standard basis vector $\pt{e}_{d+1}$.
That is, $\Pyr(\QQ) \deftobe \conv( \setof{ (\pt{x}, 0) \st \pt{x}
\in \QQ } \cup \setof{ \pt{e}_{d+1} } )$.  This operation may be
iterated, yielding the \emph{$i$-fold pyramid} $\Pyr^{i}(\QQ)
\deftobe \Pyr(\Pyr^{i-1}(\QQ)) \subset \R^{d+i}$.  (Of course,
$\Delta^{0}(\QQ) \deftobe \QQ$.)
\begin{proposition}[\protect{\cite[Proposition 3.1]{McARoc2018}}]
\label{prop:PentagonEquivToSegment-Pyramids}
   Let $P$ and $\ell$ be the pentagon and line segment defined
   above.  Then, for $i \ge 0$,
   \begin{equation}
   \label{eq:PentagonEquivToSegment-Pyramids}
      \ehr_{\Pyr^{i}(P)}(t) \equiv -\ehr_{\Pyr^{i}(\ell)}(t).
   \end{equation}
\end{proposition}

Note that the $i$-fold pyramid $\Pyr^{i}(\ell) \subset \R^{i+1}$
over $\ell$ is the simplex
\begin{equation*}
   \Pyr^{i}(\ell) = \conv\setof{\pt{0}, -\tfrac{1}{p}\pt{e}_{1},
   \pt{e}_{2}, \dotsc, \pt{e}_{i+1}}.
\end{equation*}
An important fact for the constructions below is that the period
sequence of $\Pyr^{i}(\ell)$ is $(p, 1, \dotsc, 1)$.
\begin{theorem}[\protect{\cite[Theorem 2]{BecSamWoo2008}}]
\label{thm:BecSamWoo2008}
   Let $\ell \deftobe [-\frac{1}{p}, 0]$.  Then the period
   sequence of $\Delta^{i}(\ell)$ is the $(i+2)$-tuple $(p, 1,
   \dotsc, 1)$.
\end{theorem}

\section{Nonconvex polytopes with
period sequence \texorpdfstring{$(1, \dotsc, 1, p, 1, \dotsc, 
1)$}{(1, ..., 1, p, 1, ..., 1)}}
\label{sec:11p11}

In this section, we construct an $n$-dimensional nonconvex
rational polytope $Q_{i}$ for which all Ehrhart coefficients are
constants, except for the coefficient of $t^{i}$, which has period
$p$, for arbitrary integers $p \ge 1$ and $0 \le i \le n-1$.

In the case where $i=0$, it suffices to set $Q_{0} \deftobe
\Delta^{n-1}(\ell)$ by Theorem~\ref{thm:BecSamWoo2008}.
Furthermore, we settled the $n = 2$ case in \cite{McAMor2017}.  We
thus proceed with the assumption that $i \ge 1$ and $n \ge 3$.

As in Section~\ref{sec:BuildingBlocks}, let $\ell \deftobe
[-\frac{1}{p}, 0]$ and let $P$ be the pentagon defined in terms of
$p$ by equations \eqref{eq:pentagon}.  Consider the
$n$-dimensional polytope\footnote{We adopt the natural conventions
to deal with the extreme cases $i = 1$ and $i = n-1$.  Hence,
$R_{1} \deftobe \Delta^{n-2}(P) + \pt{e}_{2}$, $R'_{1} \deftobe
\conv \setof{\pm q \pt{e}_{1}, \pt{e}_{3}, \dotsc, \pt{e}_{n}} +
\pt{e}_{2}$, $L'_{n-1} \deftobe \parens{C_{n-1} \times \setof{0}}
- \pt{e}_{n}$, and $R'_{n-1} \deftobe \parens{C_{n-2} \times
\conv{\setof{\pm q \pt{e}_{1}}}} + \pt{e}_{n}$.}
\begin{equation}\label{eq:LeftSummand}
   L_{i} \deftobe \parens{C_{i} \times \Delta^{n-i-1}(\ell)} -
   \pt{e}_{i+1}
\end{equation}
and its facet
\begin{equation}\label{eq:LeftSummandFacet}
   L'_{i} \deftobe \parens{C_{i} \times \conv\setof{\pt{0},
   \pt{e}_{2}, \dotsc, \pt{e}_{n-i}}} - \pt{e}_{i+1},
\end{equation}
as well as the $n$-dimensional
polytope\begin{equation}\label{eq:RightSummand} R_{i} \deftobe
\parens{C_{i-1} \times \Delta^{n-i-1}(P)} + \pt{e}_{i+1}
\end{equation}
and its facet
\begin{equation}\label{eq:RightSummandFacet}
   R'_{i} \deftobe \parens{C_{i-1} \times \conv\setof{q\pt{e}_{1},
   -q\pt{e}_{1}, \pt{e}_{3}, \dotsc, \pt{e}_{n-i+1}}} +
   \pt{e}_{i+1}.
\end{equation}
(Recall from Section~\ref{sec:BuildingBlocks} that $q \deftobe
p^{2}-p+1$.)  Observe that $L'_{i}$ and $R'_{i}$ are
$(n-1)$-dimensional integral polytopes in $\R^{n}$, with $L'_{i}$
contained in the hyperplane $x_{i+1} = -1$ and with $R'_{i}$
contained in the hyperplane $x_{i+1} = 1$.  Furthermore, $L_{i}$
is contained in the halfspace $x_{i+1} \le -1$, and $R_{i}$ is
contained in the halfspace $x_{i+1} \ge 1$.  Let $M_{i} \deftobe
\conv(L'_{i} \cup R'_{i})$.  Then $M_{i}$ is an $n$-dimensional
integral polytope lying between the hyperplanes $x_{i+1} = -1$ and
$x_{i+1} = 1$.  In particular, $L_{i}$, $M_{i}$, and $R_{i}$ have
pairwise disjoint interiors and integral intersections.

We are now ready to construct the not-necessarily-convex polytope
$Q_{i}$ with period sequence $(1, \dotsc, 1, p, 1, \dotsc, 1)$.
We define
\begin{equation*}
   Q_{i} \deftobe L_{i} \cup M_{i} \cup R_{i}.
\end{equation*}

\begin{theorem}\label{thm:11p11}
   Fix an arbitrary dimension $n$, degree $i$ with $0 \le i \le
   n-1$, and period $p$.  Then the $n$-dimensional polytopal ball
   $Q_{i}$ constructed above has an Ehrhart quasi-polynomial
   $\ehr_{Q_{i}}(t)$ in which all coefficient functions are
   constants, except for the coefficient of $t^{i}$, which has
   period $p$.
%    has period sequence $(1, \dotsc, 1, p, 1,
%    \dotsc, 1)$, where the period-$p$ coefficient in
%    $\ehr_{Q_{i}}(t)$ is the coefficient of $t^{i}$.
\end{theorem}

\begin{proof}
   As indicated at the beginning of this section, we may assume
   that $i \ge 1$ and $n \ge 3$.  Since $M_{i}$ is an integral
   polytope meeting $L_{i}$ and $R_{i}$ at integral facets, it
   follows from the construction above that
   \begin{equation*}
      \ehr_{Q_{i}}(t) \equiv \ehr_{L_{i}}(t) + \ehr_{R_{i}}(t).
   \end{equation*}
   Thus, 
   \begin{align*}
      \ehr_{Q_{i}}(t) %
       &\equiv \ehr_{C_{i} \times \Delta^{n-i-1}(\ell)}(t) +
       \ehr_{C_{i-1} \times \Delta^{n-i-1}(P)}(t) \\
       &= \ehr_{C_{i}}(t)  \ehr_{\Delta^{n-i-1}(\ell)}(t) + 
       \ehr_{C_{i-1}}(t)  \ehr_{\Delta^{n-i-1}(P)}(t) \\
       &\equiv \parens{\Vol(C_{i})t^{i} + \ehr_{C_{i-1}}(t)} 
       \ehr_{\Delta^{n-i-1}(\ell)}(t) - \ehr_{C_{i-1}}(t) 
       \ehr_{\Delta^{n-i-1}(\ell)}(t)\\
       &= \Vol(C_{i})t^{i}\ehr_{\Delta^{n-i-1}(\ell)}(t).       
   \end{align*}
   In this sequence of computations, the third line is the crucial
   step invoking Liu's Theorem~\ref{thm:Liu}, as well as
   Proposition \ref{prop:PentagonEquivToSegment-Pyramids}.  The
   second line uses the general fact that $\ehr_{\P \times
   \RR}(t) = \ehr_{\P}(t) \ehr_{\RR}(t)$.
   
   By Theorem~\ref{thm:BecSamWoo2008}, all coefficient functions
   in $\ehr_{\Delta^{n-i-1}(\ell)}(t)$ are constants, except for
   the ``constant'' coefficient---that is, the coefficient
   function in the degree-$0$ term---which has period $p$.  Thus,
   all coefficient functions in
   $\Vol(C_{i})t^{i}\ehr_{\Delta^{n-i-1}(\ell)}(t)$ are constant
   functions, except for the coefficient of $t^{i}$, which has
   period $p$.  It follows from the equivalence of
   quasi-polynomials shown above that the same is true of
   $\ehr_{Q_{i}}(t)$, as desired.
\end{proof}

\section{Nonconvex polytopes with period sequence
\texorpdfstring{$(p_{0}, p_{1}, \dotsc, p_{n-1}, 1)$}{(p\_0, p\_1,
..., p\_\{n-1\}, 1)}}
\label{sec:pppp1}

In this section, we construct a nonconvex polytope $Q_{\ast}$ for
which the Ehrhart quasi-polynomial has an arbitrary period
sequence.  Let the desired period sequence be $(p_{0}, p_{1},
\dotsc, p_{n-1}, 1)$, where each $p_{i}$ is a positive integer.
The previous section showed how to construct a polytope $Q_{i}$
with period sequence $(1, \dotsc, 1, p_{i}, 1, \dotsc, 1)$, where
$p_{i}$ is the period of the $i$\textsuperscript{th} coefficient.
In this section, we will modify that construction so that the
resulting ``modified $Q_{i}$'' can be glued together to build the
polytope $Q_{\ast}$ with the desired period sequence.

For $0 \le i \le n-1$, let $\ell_{i} \deftobe [-\frac{1}{p_{i}},
0]$, and let $P_{i}$ be the pentagon defined by equations
\eqref{eq:pentagon} after replacing $p$ by $p_{i}$ and $q$ by
$q_{i} \deftobe p_{i}^{2} - p_{i} + 1$.

We first deal with the periods $p_{i}$ with $i \ge 1$.  Observe
that the construction of $Q_{i}$ in Section~\ref{sec:11p11} goes
through if we replace the translations by $\pm\pt{e}_{i+1}$ in
% the definitions
% of $L_{i}$, $R_{i}$ and their respective facets $L'_{i}$, $R'_{i}$
equations~\eqref{eq:LeftSummand}--\eqref{eq:RightSummandFacet}
with translations by $\pm k_{i} \pt{e}_{i+1}$, where $k_{i}$ is an
arbitrary positive integer (to be fixed below), as follows:
\begin{align}
   L_{i} & \deftobe \parens{C_{i} \times \Delta^{n-i-1}(\ell_{i})}
   - k_{i}\pt{e}_{i+1}, \label{eq:LeftSummandNew} \\
   L'_{i} & \deftobe \parens{C_{i} \times \conv\setof{\pt{0},
   \pt{e}_{2}, \dotsc, \pt{e}_{n-i}}} - k_{i}\pt{e}_{i+1}, \\
   R_{i} & \deftobe \parens{C_{i-1} \times \Delta^{n-i-1}(P_{i})} +
   k_{i}\pt{e}_{i+1}, \\
   R'_{i} & \deftobe \parens{C_{i-1} \times
   \conv\setof{q_{i}\pt{e}_{1}, -q_{i}\pt{e}_{1}, \pt{e}_{3},
   \dotsc, \pt{e}_{n-i+1}}} + k_{i}\pt{e}_{i+1}.
   \label{eq:RightSummandFacetNew}
\end{align}
Likewise, to handle the periodicity $p_{0}$ in the degree-$0$
term, we may define a translated version of the polytope $Q_{0} $
from Section~\ref{sec:11p11}, as well as one of its facets:
\begin{align}
   Q_{0} & \deftobe \Delta^{n-1}(\ell_{0}) - k_{0}\pt{e}_{1}, \\
   Q'_{0} & \deftobe \conv\setof{\pt{0}, \pt{e}_{2}, \dotsc,
   \pt{e}_{n}} - k_{0}\pt{e}_{1} \label{eq:PeriodZeroFacetNew}.
\end{align}

As in Section~\ref{sec:11p11}, we find that $L'_{i}$ and $R'_{i}$
are $(n-1)$-dimensional integral polytopes in $\R^{n}$, with
$L'_{i}$ contained in the hyperplane $x_{i+1} = -k_{i}$ and with
$R'_{i}$ contained in the hyperplane $x_{i+1} = k_{i}$.
Furthermore, $L_{i}$ is contained in the halfspace $x_{i+1} \le
-k_{i}$, and $R_{i}$ is contained in the halfspace $x_{i+1} \ge
k_{i}$.  Finally, $Q'_{0}$ is contained in the hyperplane $x_{1} =
-k_{0}$, and $Q_{0}$ is contained in the halfspace $x_{1} \le
-k_{0}$.

We now fix the translation parameters $k_{0}, k_{1}, \dotsc,
k_{n-1}$ in equations
\eqref{eq:LeftSummandNew}--\eqref{eq:PeriodZeroFacetNew} to be
sufficiently large so that the vertices of the facets $Q'_{0},
L'_{1}, R'_{1}, \dotsc, L'_{n-1}, R'_{n-1}$ all lie in convex
position.  That is, we choose $k_{0}, k_{1}, \dotsc, k_{n-1}$ so
that each of these facets is a facet of the integral polytope
$$M \deftobe \conv \parens{Q'_{0} \cup L'_{1} \cup R'_{1} \cup
\dotsb \cup L'_{n-1} \cup R'_{n-1}}.$$ (The values of $k_{i}$ that
are sufficiently large will depend upon the desired periods
$p_{0}, p_{1}, \dotsc, p_{n-1}$, as well as on the particular
$(n+1)$-subset $T \subset \Z$ used to construct the cyclic
polytopes $C_{i}$ in Section~\ref{sec:CyclicPolytopes}.)

Using these translated versions for $L_{i}$ and $R_{i}$, let us
redefine the polytope $Q_{i}$ for $1 \le i \le n-1$ by setting
$Q_{i} \deftobe L_{i} \cup M \cup R_{i}$.  As in
Section~\ref{sec:11p11}, $Q_{i}$ has period sequence $(1, \dotsc,
1, p_{i}, 1, \dotsc, 1)$.  Finally, we let $Q_{\ast} \deftobe
\bigcup_{i=0}^{n-1} Q_{i}$.  As in the proof
of~Theorem~\ref{thm:11p11}, we compute that
\begin{equation*}
   \ehr_{Q_{\ast}}(t) \equiv \sum_{i=0}^{n-1}
   \Vol(C_{i})t^{i}\ehr_{\Delta^{n-i-1}(\ell_{i})}(t).
\end{equation*}
Therefore, the period sequence of $Q_{\ast}$ is $(p_{0}, p_{1},
\dotsc, p_{n-1}, 1)$, as desired.

\def\cprime{$'$}
\providecommand{\bysame}{\leavevmode\hbox to3em{\hrulefill}\thinspace}
\providecommand{\MR}{\relax\ifhmode\unskip\space\fi MR }
\providecommand{\MRhref}[2]{%
  \href{http://www.ams.org/mathscinet-getitem?mr=#1}{#2}
}
\providecommand{\href}[2]{#2}

\end{document}